\documentclass{article}
\usepackage[centertags]{amsmath}
\usepackage{amsfonts}
\usepackage{amssymb}
\usepackage{amsthm}
\usepackage{makeidx}
\newtheorem{thm}{Theorem}[section]

\theoremstyle{definition}

\theoremstyle{definition}

\theoremstyle{definition}
\newtheorem{rem}{Remark}[section]
\numberwithin{equation}{section}

\newcommand{\punt}{\boldsymbol{.}}
\begin{document}
\author{P. Petrullo and D. Senato}
\date{}
\title{An instance of umbral methods in representation theory: the parking function module}
\maketitle
\thispagestyle{empty}
\begin{center}
\textsf{Dipartimento di Matematica e Informatica, Universit\'a
degli Studi della Basilicata, via dell'Ateneo Lucano 10, 85100
Potenza, Italia}.\\
\verb"p.petrullo@gmail.com, domenico.senato@unibas.it"
\end{center}
\textsf{\textbf{keywords}: parking functions, noncrossing
partitions, volume polynomial, umbral calculus, Abel polynomials.}
\begin{abstract}
We test the umbral methods introduced by Rota and Taylor within
the theory of representation of the symmetric group. We prove that
the volume polynomial of Pitman and
Stanley represents the Frobenius characteristic of the Haiman parking function module, when the set of its variables consists of suitable umbrae. We also show that the volume polynomial in any set of similar and uncorrelated umbrae is umbrally equivalent, up to a constant
term, to an
Abel-like umbral polynomial. An
analogous treatment of the parking function module of type B is given.

\end{abstract}

\section{Introduction}
Parking functions were introduced by Konheim and Weiss
\cite{KonWei} in the Sixties. Afterwards several authors gave a
strong contribution to the development of this subject in the
context of combinatorics and representation theory
\cite{FoaRio},\cite{KunYan}, \cite{NovThy1,NovThy2}, \cite{Stan3}.
Recently Pitman and Stanley \cite{PitStan} have introduced the so-called volume polynomial, that arises naturally in several different settings, in particular in the study of plane partitions and parking functions. Haiman \cite{Haim} has defined the parking
function module considering the standard action of the symmetric
group $\mathfrak{S}_n$ on the set of all parking functions of
length $n$. In this context, parking functions give a simple
combinatorial description of the ring $R_n$ of diagonal
coinvariants of $\mathbb{C}[\mathbf{x},\mathbf{y}]$, seen as a
representation of the symmetric group. See \cite{Haim2} for a
survey.\\
The systematic study of noncrossing partitions was started in the
Seventies by Kreweras \cite{Krew} and Poupard \cite{Pou}. This
subject arises in the wide-ranging of connections between algebra
and combinatorics \cite{Edel1,Edel2}, \cite{EdelSim}, \cite{Sim1},
\cite{Spei}, an overview can be found in \cite{Sim2}.
In detail, the noncrossing partition lattice turns out to have
strong symmetry properties which give rise to a symmetric function
corresponding to a representation of the symmetric group. Indeed,
Stanley \cite{Stan1} has recovered the Haiman action as a local
action of the symmetric group on the maximal chains of the lattice
of noncrossing partitions. This is done by defining an edge
labelling of maximal chains with parking functions, each occuring
once.\\
As shown by Biane \cite{Bia1} the lattice of noncrossing
partitions can be embedded into the Cayley graph of the symmetric
group. Moreover, Reiner
\cite{Rei} has introduced a class of noncrossing partitions for
all classical reflection groups. In particular, type A reflection groups (i.e. symmetric groups) correspond to classical noncrossing partitions. The hyperoctahedral groups, that is type B reflection groups, correspond to the \textit{noncrossing partitions of type B}, that is
the noncrossing partitions of the set $[\pm
n]=\{-n,\ldots,-1,1,\ldots,n\}$ which are invariant under sign
change. The results of Biane and Reiner have been generalized by
Brady \cite{Br}, Brady and Watt \cite{BW} and Bessis \cite{Bes}. A
poset $\mathcal{NC}(W)$ of noncrossing partitions can be defined
for finite Coxeter systems $(W,S)$ of each type. In particular,
these posets have a nice description in terms of the length
function $\ell_{\scriptscriptstyle T}$ defined with respect to their
sets $T$ of reflections. The reader may refer \cite{Arm} for an overview on the combinatorial
aspect of this beautiful subject. The notion of parking function of type
B, introduced by Stanley \cite{Stan1} and deeply studied by Biane
\cite{Bia}, parallels the classical one providing an edge
labelling of maximal chains in the lattice of noncrossing
partitions of type B.\\
The aim of this paper is to test the umbral methods introduced in
Rota and Taylor \cite{RotTay} in this context. Applications of
these methods are given by Zeilberger \cite{Zei}, where generating
functions are computed for many difficult problems dealing with
counting many combinatorial objects. Applications to bilinear
generating functions for polynomial sequences are given by Gessel
\cite{Ges}. The ideas of Rota and Taylor have been developed by Di
Nardo and Senato in \cite{DiNaSen1,DiNaSen2} and here we follow
this last point of view. In this paper, we prove that the
$n$-volume polynomial $V_n(x_1,\ldots,x_n)$ of Pitman and Stanley
represents the Frobenius characteristic $\mathcal{PF}_n$ of the
Haiman parking function module, when each variable $x_i$ is
replaced by a suitable umbra $\bar{\vartheta}_i$. We also use Abel
polynomials $A_n(x,\alpha)=x(x-n\punt\alpha)^{n-1}$ of Rota, Shen and Taylor
\cite{RST} to show that
$n!V_n(\bar{\vartheta}_1,\ldots,\bar{\vartheta}_n)$ and
$\bar{\vartheta}(\bar{\vartheta}+n\punt\bar{\vartheta})^{n-1}$ are
umbrally equivalent. Analogous treatment is reserved to the parking function module of type B. In this case, the polynomials
$B_n(x,\alpha)=(x-n\punt\alpha)^{n}$  play a role of Abel polynomials of type B.
Further applications of Abel polynomial can be found in \cite{DNPS},
where an unifying framework for the cumulant theory, both
classical, boolean and free, is given.
\section{Parking functions, noncrossing partitions and volume polynomial}

A \textit{parking function} of length $n$ is a sequence
$\mathbf{p}=(p_1,\ldots,p_n)$ of $n$ positive integers whose
nondecreasing arrangement
$\mathbf{p}^{\prime}=(p^{\prime}_{1},\ldots,p^{\prime}_{n})$ is
such that $p^{\prime}_{j}\leq j$. As in {\cite{PitStan}} we denote
by $park(n)$ the set of all parking functions of length $n$. Its
cardinality is $(n+1)^{n-1}$. The symmetric group $\mathfrak{S}_n$
acts on the set $park(n)$ by permuting the entries of parking
functions (\textit{standard action}). As introduced by Haiman
{\cite{Haim}}, the \textit{parking function module} is obtained by
considering the standard action of $\frak{S}_n$ on the
$\mathbb{Q}$-vector space spanned by all parking functions of length
$n$. The number of orbits of this action is equal to the $n$-th
Catalan number $C_n$, that is
\begin{displaymath}
C_n=\frac{1}{n+1}{2n \choose n}.
\end{displaymath}
The Frobenius characteristic of the parking function module (i.e.
the symmetric function associated to its character by the
Frobenius map $ch$) is known as \textit{parking function symmetric
function} and is denoted by $\mathcal{PF}_n$. We have
\begin{equation}\label{id:PFn}
\mathcal{PF}_n=\sum_{\mu\vdash
n}\frac{(n)_{\ell(\mu)-1}}{m({\mu})!}h_{\mu},
\end{equation}
where the sum ranges over all integer partitions $\mu$ of $n$.
More precisely, in (\ref{id:PFn}) we have $m(\mu)!=m_1!\cdots
m_n!$, $m_i$ being the number of parts of $\mu$ equal to $i$,
$\ell(\mu)=m_1+\cdots+m_n$, and $h_{\mu}$ denotes the complete
homogeneous symmetric function indexed by $\mu$. There are at
least three other ways to get the symmetric function
$\mathcal{PF}_n$. Two of these arise within noncrossing
partitions and are due to Stanley \cite{Stan1}.\\
Let $[n]$ denote the set $\{1,\ldots,n\}$ of positive integers. A
\textit{noncrossing partition} of $[n]$ is a partition
$\pi=\{B_1,\ldots,B_s\}$ of $[n]$, such that, if $1\leq
h<l<k<m\leq n$ with $h,k\in B_j$ and $l,m\in B_{j'}$, then $j=j'$.
As usual, denote by $\mathcal{NC}_n$ the set of all noncrossing
partitions of $[n]$. Its cardinality is $|\mathcal{NC}_n|=C_n$
too. A simple bijection between noncrossing partitions of $[n]$
and orbits of the parking function module was given by Rattan \cite{Rat}.\\
As well known, $\mathcal{NC}_n$ is a lattice of rank
$n-1$ with respect to the refinement order. Denote by
$\mathbf{0}_n$ and $\mathbf{1}_n$ its minimum and maximum element
respectively. The number of maximal chains of $\mathcal{NC}_{n}$
is $n^{n-2}$. As shown by Stanley \cite{Stan1}, maximal chains of
$\mathcal{NC}_{n+1}$, whose number is $(n+1)^{n-1}$, can be
labelled by parking functions, each occurring once. Moreover, if
$V_{\scriptscriptstyle\mathcal{NC}_{n+1}}$ is the
$\mathbb{Q}$-vector space spanned by all the maximal chains of $\mathcal{NC}_{n+1}$, a local action of the symmetric group $\frak{S}_n$ can be defined
on $V_{\scriptscriptstyle\mathcal{NC}_{n+1}}$ which turns out to
have the same character of the parking function module. In
particular, this is obtained by transferring the action of
$\frak{S}_n$ on parking functions to their respective maximal
chains. Stanley has also proved the following generalization. Let $k$ be a positive integer. A $k$-parking
function is a sequence $\mathbf{p}=(p_1,\ldots,p_n)$ of $n$
positive integers whose nondecreasing arrangement
$\mathbf{p}^{\prime}=(p^{\prime}_{1},\ldots,p^{\prime}_{n})$ is
such that $p^{\prime}_{j}\leq kj$. Let $\mathcal{NC}^{(k)}_n$
denote the subset of all noncrossing partitions of
$\mathcal{NC}_{kn}$ whose block's cardinalities are multiples of
$k$. Maximal chains of $\mathcal{NC}_{n}^{(k)}$ are labelled by
$k$-parking functions each occurring once. This yields a local
action of $\frak{S}_n$ on the $\mathbb{Q}$-vector space spanned by
the maximal chains of $\mathcal{NC}_{n}^{(k)}$ which is isomorphic
to the $k$-parking function module. The following results are proved
in \cite{Stan1}.
\begin{thm}
If $H(t)=1+\sum_{n\geq 1}h_n t^n$ and
$\mathcal{PF}(t)=1+\sum_{n\geq 1}\mathcal{PF}_nt^n$, then we have
\begin{equation}\label{St1}
t\mathcal{PF}(t)=\left[\frac{t}{H(t)}\right]^{\scriptscriptstyle<-1>},
\end{equation}
where $^{\scriptscriptstyle<-1>}$ denotes the compositional
inverse. More generally, if $k$ is a positive
integer, $\mathcal{PF}_n^{(k)}$ is the Frobenius characteristic of
the $k$-parking function module and
$\mathcal{PF}^{(k)}(t)=1+\sum_{n\geq 1}\mathcal{PF}_n^{(k)}t^n$,
then
\begin{equation}\label{St2}
t\mathcal{PF}^{(k)}(t)=\left[\frac{t}{H(t)^k}\right]^{\scriptscriptstyle<-1>}.
\end{equation}
\end{thm}
A second way to obtain the parking function symmetric function
is the following. For each $S\subseteq[n-1]$ the \textit{Gessel's
quasi-symmetric function} $\mathcal{Q}_S$ is defined by
$$\mathcal{Q}_S=\sum_{i_1\leq\ldots\leq i_n\atop i_j<i_{j+1}\text{ if }j\in
S}x_{i_1}\cdots x_{i_n}.$$
Let $r$ denote the rank function of the lattice
$\mathcal{NC}_{n+1}$. If $S\subseteq[n-1]$ and $|S|=s-1$, then let
$\alpha_{\scriptscriptstyle\mathcal{NC}_{n+1}}(S)$ be the number
of chains
$\mathbf{0}_{n+1}=\pi_0<\pi_1<\ldots<\pi_s=\mathbf{1}_{n+1}$ such
that $S=\{r(\pi_1),\ldots,r(\pi_{s-1})\}$. Define
$\beta_{\scriptscriptstyle\mathcal{NC}_{n+1}}(S)$ to be the
following integer:
$$\beta_{\scriptscriptstyle\mathcal{NC}_{n+1}}(S)=\sum_{T\subseteq S}(-1)^{|S-T|}\alpha_{\scriptscriptstyle\mathcal{NC}_{n+1}}(S).$$
The functions $\alpha_{\scriptscriptstyle\mathcal{NC}_{n+1}}$ and
$\beta_{\scriptscriptstyle\mathcal{NC}_{n+1}}$ are named
\textit{flag f-vector} and \textit{flag h-vector} respectively, of
$\mathcal{NC}_{n+1}$. The connection between
$\beta_{\scriptscriptstyle\mathcal{NC}_{n+1}}$ and $\mathcal{PF}_n$ is
shown by the following theorem.
\begin{thm}[\cite{Stan1}]
Let $\beta_{\scriptscriptstyle\mathcal{NC}_{n+1}}$ and
$\mathcal{Q}_{S}$ be defined as above. Then the polynomial
$$F_{\scriptscriptstyle\mathcal{NC}_{n+1}}=\sum_{S\subseteq[n-1]}\beta_{\scriptscriptstyle\mathcal{NC}_{n+1}}(S)\mathcal{Q}_S,$$
is symmetric and it is such that
\begin{equation}\label{id:flag}
\omega F_{\scriptscriptstyle\mathcal{NC}_{n+1}}=\mathcal{PF}_n,
\end{equation}
$\omega$ being the involution of the ring of the symmetric
functions mapping complete homogeneous symmetric functions $h_n$
onto elementary symmetric functions $e_n$.
\end{thm}
Third approach to $\mathcal{PF}_n$ is essentially
based on an involution $\psi$ on the ring of symmetric functions
defined by Macdonald \cite{Mac} (see also \cite{Haim} and
\cite{Len} on this subject). The map $\psi$ is defined by
$\psi(h_n)=h_n^*$, where $h_n^*$ are symmetric functions whose
generating function $H^*(z)=1+\sum_{n\geq 1}h_n^*z^n$ has the following
property:
$$zH^*(z)=[zH(z)]^{\scriptscriptstyle<-1>}.$$
Lagrange inversion gives
$$(-1)^nh_n^*=\sum_{\mu\vdash n}\frac{(n)_{\ell(\mu)-1}}{m(\mu)!}e_{\mu},$$
thus, by virtue of (\ref{id:PFn}) we have
\begin{equation}\label{id:star}
(-1)^n\omega(h_n^*)=\mathcal{PF}_n.
\end{equation}
Let us recall the notion of volume polynomial. If
$\mathbf{x}=\{x_1,\ldots,x_n\}$ is a set of commuting variables,
and $\mathbf{c}=(c_1,\ldots,c_l)$ is a sequence of positive
integers with $l\leq n$, then we set
\begin{eqnarray}
\mathbf{x}_{\mathbf{c}}&=&x_{c_1}\cdots
x_{c_l},\nonumber\\
\mathbf{x}^{\mathbf{c}}&=&x_1^{c_1}\cdots x_n^{c_n}\nonumber.
\end{eqnarray}
Following Pitman and Stanley \cite{PitStan}, we define the
$n$-\textit{volume polynomial} in the set of variables
$\mathbf{x}$ to be the polynomial
$V_n(\mathbf{x})=V_n(x_1,\ldots,x_n)$ such that
\begin{equation}\label{def:park-polin}
V_n(\mathbf{x})=\frac{1}{n!}\!\!\sum_{\scriptscriptstyle
\mathbf{p}\in park(n)}\mathbf{x}_{\mathbf{p}}.
\end{equation}
Straightforward computations provides the following expression of
$V_n(\mathbf{x})$:
\begin{equation}\label{id:park-pol}
V_n(\mathbf{x})=\sum_{\mu\vdash
n}\frac{(n)_{\ell(\mu)-1}}{m(\mu)!\mu!}\mathbf{x}^{\mu},
\end{equation}
where $\mathbf{x}^{\mu}=x_1^{\mu_1}\cdots x_l^{\mu_l}$ and
$\mu!=\mu_1!\cdots\mu_l!$ whenever $\mu=(\mu_1,\ldots,\mu_l)$.
\begin{rem}\label{rem}
We stress a direct connection between $\mathcal{PF}_n$ and this
last expression of $V_n(\mathbf{x})$: by making the
symbolic substitution
$\mathbf{x}^{\mu}\rightarrow\mathbf{x}_{\mu}$ in
(\ref{id:park-pol}), and then setting $x_i=i!h_i$, we recover the
parking function symmetric function $\mathcal{PF}_n$. This fact
suggests us the introduction of umbral notations.
\end{rem}
\section{Umbrae and Abel polynomials}
\textit{Classical umbral calculus} is a strongly symbolic method
for the manipulation of sequences $(1,a_1,a_2,\ldots)$, where
$a_i$ belongs to some ring $R$ whose quotient field is of
characteristic zero. It essentially consists of the following
data:
\begin{enumerate}
\item a set $A=\{\alpha,\gamma,\delta,\ldots\}$, called the
\textit{alphabet}, whose elements are named \textit{umbrae},
\item a linear functional $E$, called \textit{evaluation}, defined
on the polynomial ring $R[A]$ and taking value in $R$, such that
\begin{itemize}
\item $E[1]=1$,
\item
$E[\alpha^i\gamma^j\cdots\delta^k]=E[\alpha^i]E[\gamma^j]\cdots
E[\delta^k]$ for all pairwise distinct umbrae $\alpha,\gamma,\ldots,\delta$ (\textit{uncorrelation property}),
\end{itemize}
\item two special umbrae $\varepsilon$ (\textit{augmentation}) and
$u$ (\textit{unity}) such that
\begin{displaymath}
E[\varepsilon^i]=\delta_{0,i},\quad\textrm{for }i=0,1,2,\ldots,
$$and$$
E[u^i]=1,\quad\textrm{for }i=0,1,2,\ldots.
\end{displaymath}
\end{enumerate}
A sequence $(1,a_1,a_2,\ldots)$ is said to be \textit{represented}
by an umbra $\alpha$ if $E[\alpha^i]=a_i$ for $i=1,2,\ldots$ (note
that $E[\alpha^0]=1$ for all $\alpha$). In this case we say $a_i$
is the $i$-th \textit{moment} of $\alpha$. Two umbrae $\alpha$ and
$\gamma$ are said to be \textit{umbrally equivalent}, denoted by
$\alpha\simeq\gamma$, if $E[\alpha]=E[\gamma]$. They are said to
be \textit{similar}, written $\alpha\equiv\gamma$, if $\alpha$ and
$\gamma$ represent the same sequence of moments, that is
$\alpha^i\simeq\gamma^i$ for all $i=1,2,\ldots$. We can extend
coeffincientwise the action of $E$ to exponential formal power
series
\begin{displaymath}
e^{\alpha t}=\sum_{i\geq 0}\alpha^i\frac{t^i}{i!}
\end{displaymath}
obtaining in this way the \textit{generating function}
$f(\alpha,t)$ of $\alpha$:
\begin{displaymath}
f(\alpha,t)=E[e^{\alpha t}]=\sum_{i\geq
0}E[\alpha^i]\frac{t^i}{i!}=1+\sum_{i\geq 1}a_i\frac{t^i}{i!}.
\end{displaymath}
Note that $\alpha\equiv\gamma$ if and only if
$f(\alpha,t)=f(\gamma,t)$. The generating functions of the augmentation $\varepsilon$ and the
unity $u$ are respectively
\begin{displaymath}
f(\varepsilon,t)=E[e^{\varepsilon t}]=\sum_{i\geq
0}E[\varepsilon^i]\frac{t^i}{i!}=1
$$and$$
f(u,t)=E[e^{u t}]=\sum_{i\geq 0}E[u^i]\frac{t^i}{i!}=1+\sum_{i\geq
1}\frac{t^i}{i!}=e^t.
\end{displaymath}
If $\alpha$ and $\gamma$ are two
umbrae, then the generating function $f(\alpha+\gamma,t)$ is given
by $f(\alpha,t)f(\gamma,t)$. In fact
\begin{displaymath}
f(\alpha+\gamma,z)=E[e^{(\alpha+\gamma)t}]=E[e^{\alpha t}e^{\gamma
t}]=E[e^{\alpha t}]E[e^{\gamma t}]=f(\alpha,z)f(\gamma,z).
\end{displaymath}
The \textit{Bell umbra} $\beta$ is defined to be an umbra
representing the sequence of Bell numbers $\mathcal{B}_i$, that is
$\beta^i\simeq\mathcal{B}_i$. In this way
\begin{displaymath}
f(\beta,t)=e^{e^t-1}.
\end{displaymath}
The singleton umbra $\chi$ has moments $\chi^i\simeq 1$ if
$i=0,1$, and $\chi^i\simeq 0$ otherwise. Its generating function
is
\begin{displaymath}
f(\chi,t)=1+t.
\end{displaymath}
We work with a \textit{saturated umbral calculus}, see
\cite{RotTay}, if we extend the action of the evaluation $E$ to
the ring $R[A\cup B]$, where $B$ is the \textit{auxiliary
alphabet} whose elements, named \textit{auxiliary umbrae}, are
defined starting from the umbrae in $A$. Umbral equivalence and
similarity are extended via $E$ to polynomials $p$ and $q$ in
$R[A\cup B]$. Given $\alpha\in A$, first auxiliary umbra we
introduce is denoted by $-1\punt\alpha$. It is uniquely determined
(up to similarity) by the condition
\begin{equation}\label{def:-1.alpha}
\alpha+(-1\punt\alpha)\equiv\varepsilon.
\end{equation}
Its generating function is $f(\alpha,t)^{-1}$. Indeed, from
(\ref{def:-1.alpha}) we have
\begin{displaymath}
1=f(\varepsilon,t)= f[\alpha+(-1\punt\alpha),t]= f(\alpha,t)
f(-1\punt\alpha,t).
\end{displaymath}
More generally, if $n$ is an integer, the umbra denoted by
$n\punt\alpha$ is such that
$$n\punt\alpha\equiv\alpha_1+\cdots+\alpha_n,$$
$\alpha_1,\ldots,\alpha_n$ being uncorrelated umbrae similar to
$\alpha$. We have
$$f(n\punt\alpha,t)=f(\alpha,t)^n.$$
As introduced by Rota, Shen and Taylor \cite{RST}, an Abel
polynomials in the variable $x$ is a polynomial in $R[A\cup B][x]$
of type
$$A_{n}(x,\alpha)=x(x-n\punt\alpha)^{n-1}.$$
The following theorem states that
$n!V_n(\alpha_1,\ldots,\alpha_n)$ and
$A_{n}(\alpha,-1\punt\alpha)$ are in the same class of umbral
equivalence for all $n\geq 1$.
\begin{thm}
Let $\alpha$ be an umbra, $\alpha_1,\ldots,\alpha_n$ be $n$
uncorrelated umbrae similar to $\alpha$, and
$V_n(\alpha_1,\ldots,\alpha_n)$ be the $n$-volume polynomial
(\ref{def:park-polin}) in $x_i=\alpha_i$. Then we have
\begin{equation}\label{id:thm}
n!V_n(\alpha_1,\ldots,\alpha_n)\simeq
A_{n}(\alpha,-1\punt\alpha)\simeq\alpha(\alpha+n\punt\alpha)^{n-1}.
\end{equation}
\end{thm}
\begin{proof}
We denote by $a_i$ the $i$-th moment of $\alpha$. Then, the $k$-th
moment of $n\punt\alpha$ is given by
\begin{displaymath}
(n\punt\alpha)^{k}\simeq\sum_{\lambda\vdash
k}\mathrm{d}_{\lambda}\,(n)_{\ell(\lambda)}\,a_{\lambda},
\end{displaymath}
where $\mathrm{d}_{\lambda}=k!/(\lambda!m(\lambda)!)$ and
$a_{\lambda}=a_{\lambda_1}\cdots a_{\lambda_l}$, whenever
$\lambda=(\lambda_1,\ldots,\lambda_l)$. Thus, we have
\begin{displaymath}
\alpha(\alpha+n\punt\alpha)^{n-1}\simeq\sum_{1\leq k\leq
n}\sum_{\lambda\vdash n-k}{n-1\choose
k-1}\mathrm{d}_{\lambda}\,(n)_{\ell(\lambda)}\,a_ka_{\lambda}.
\end{displaymath}
Let $Par(n)$ the set of all the integer partitions of $n$. If $S_n=\{(k,\lambda)\,|\,1\leq k\leq
n,\,\lambda\vdash n-k\}$, then the previous equivalence can be
rewritten as
\begin{displaymath}
\alpha(\alpha+n\punt\alpha)^{n-1}\simeq\sum_{(k,\lambda)\in
S_n}{n-1\choose
k-1}\mathrm{d}_{\lambda}\,(n)_{\ell(\lambda)}\,a_ka_{\lambda}.
\end{displaymath}
Let $\Delta:S_n\hookrightarrow Par(n)$ be defined by
$\Delta(k,\lambda)=k\cup\lambda$, where $k\cup\lambda$ is the
integer partition obtained adding the part $k$ to $\lambda$. If
$\Delta(k,\lambda)=\mu$ and $m(\mu)=(m_1,m_2,\ldots,m_n)$, then
$a_ka_{\lambda}=a_{\mu}$, $\ell(\lambda)=\ell(\mu)-1$,
$k!\lambda!=\mu!$ and $m(\lambda)!=m(\mu)!/m_k$, so that
\begin{displaymath}
{n-1\choose
k-1}\mathrm{d}_{\lambda}\,(n)_{\ell(\lambda)}a_ka_{\lambda}
=\frac{n!}{\mu!}(n)_{\ell(\mu)-1}\frac{k}{n}\frac{m_k}{m(\mu)!}a_{\mu}.
\end{displaymath}
Let $\{\mu\}$ denote the set of all distinct parts of $\mu$, that
is the set whose elements are the parts of $\mu$ each occurring
once. Since $\sum_{k\in\{\mu\}}km_k=n$, finally we gain
\begin{displaymath}
\alpha(\alpha+n\punt\alpha)^{n-1}\simeq n!\sum_{\mu\vdash
n}\frac{(n)_{\ell(\mu)-1}}{m(\mu)!\mu!}a_{\mu}.
\end{displaymath}
Equivalence (\ref{id:thm}) follows from (\ref{id:park-pol}) and
from the fact that if $x_i=\alpha_i$ then
$\mathbf{x}^{\mu}=\alpha_1^{\mu_1}\cdots\alpha_l^{\mu_l}\simeq
a_{\mu}$.
\end{proof}
\begin{rem}
Theorem \ref{id:thm} parallels a known result involving
$V_n(\mathbf{x})$ proved by Pitman and Stanley in \cite{PitStan}.
More precisely, for all $a\in\mathbb{C}$ we have
$$V_n(a,\ldots,a)=a(a+na)^{n-1},$$
so that $V_n(a,\ldots,a)$ is obtained by evaluating $x=a$ in the
Abel polynomial $A_{n}(x,-a)=x(x+na)^{n-1}$.
\end{rem}
We assume $R=\mathbb{Q}[\mathbf{x}]$ in the umbral setting, that is $R$ is the ring of
polynomials with rational coefficient in the set of variables
$\mathbf{x}$. Evaluation $E$ maps the umbrae of the base alphabet
$A$ in polynomials of $\mathbb{Q}[\mathbf{x}]$. For this reason we
call them \textit{polynomial umbrae}. Let $\bar{\epsilon}$ be a
polynomial umbra such that
$$
\bar{\epsilon}\equiv\chi_1 x_1+\cdots+\chi_n x_n,
$$
where $\chi_1,\ldots,\chi_n$ are $n$ uncorrelated umbrae similar
to $\chi$. Its generating function is
$$f(\bar{\epsilon},t)=\prod_{i=1}^{n}f(\chi_i x_i,t)=\prod_{i=1}^{n}(1+x_it)=1+\sum_{n\geq 1}e_it^n=E(t),$$
where $e_i=e_i(\mathbf{x})$ is the $i$-th elementary symmetric
function in the variables $\mathbf{x}$. Thus, $i$-th moment of $\bar{\epsilon}$ is given by
$$
\bar{\epsilon}^i\equiv i!e_i.
$$
In \cite{DiNaSenGua} a
polynomial umbra $\epsilon$ representing elementary symmetric
functions (that is $\epsilon^i\simeq e_i$) was defined, from which
the choice of the symbol $\bar{\epsilon}$. We define a
new polynomial umbra $\bar{\vartheta}$ as follows:
\begin{equation}\label{def:
vartheta}
\bar{\vartheta}\equiv -1\punt-\bar{\epsilon}.
\end{equation}
Since for every umbra $\alpha$ we have
$$f(-\alpha,t)=f(\alpha,-t),$$
it is clear that $f(\bar{\vartheta},t)=f(\bar{\epsilon},
-t)^{-1}=E(-t)^{-1}$. Moreover, being
$H(t)E(-t)=1$ we have $f(\bar{\vartheta},t)=H(t)$, and
$$
\bar{\vartheta}^i\simeq i!h_i.
$$
\begin{thm}
If $\bar{\vartheta}_1,\ldots,\bar{\vartheta}_n$ are $n$
uncorrelated umbrae similar to $\bar{\vartheta}$ and
$V_n(\mathbf{x})$ is the $n$-volume polynomial
$(\ref{def:park-polin})$, then
\begin{equation}\label{id:park-pol vs PFn}
V_n(\bar{\vartheta}_1,\ldots,\bar{\vartheta}_n)\simeq\mathcal{PF}_n.
\end{equation}
\end{thm}
\begin{proof}
Since
$\bar{\vartheta}_1^{\mu_1}\cdots\bar{\vartheta}_l^{\mu_l}\simeq\mu!h_{\mu}$
if $\mu_i$ are the parts of $\mu$, by virtue of identities (\ref{id:PFn}) and (\ref{id:park-pol}) we have proved the theorem.
\end{proof}
The relation between $V_n(\mathbf{x})$ and the symmetric functions
$F_{\scriptscriptstyle\mathcal{NC}_{n+1}}$ and $h_n^*$ introduced
in the previous section is stated in the following theorem.
\begin{thm}
If $\bar{\epsilon}_1,\ldots,\bar{\epsilon}_n$ are $n$ uncorrelated
umbrae similar to $\bar{\epsilon}$  and $V_n(\mathbf{x})$ is the
$n$-volume polynomial $(\ref{def:park-polin})$, then
\begin{eqnarray}
\nonumber
V_n(\bar{\epsilon}_1,\ldots,\bar{\epsilon}_n)&\simeq&F_{\scriptscriptstyle\mathcal{NC}_{n+1}},\\
\nonumber V_n(-\bar{\epsilon}_1,\ldots,-\bar{\epsilon}_n)&\simeq&h_n^*.
\end{eqnarray}
\end{thm}
\begin{proof}
Observe that
$\bar{\epsilon}_1^{\mu_1}\cdots\bar{\epsilon}_l^{\mu_l}\simeq\mu!e_{\mu}$
and
$(-\bar{\epsilon}_1)^{\mu_1}\cdots(-\bar{\epsilon}_l)^{\mu_l}\simeq(-1)^n\mu!
e_{\mu}$. The theorem is proved by comparing (\ref{id:park-pol})
with (\ref{id:flag}) and (\ref{id:star}).
\end{proof}
In order to show the connection between $V_n(\mathbf{x})$ and
$\mathcal{PF}_n^{(k)}$, the following auxiliary umbrae may be
useful. For each umbra $\alpha$, let
$\alpha^{\scriptscriptstyle<-1>}$ denote an auxiliary umbra such
that
\begin{displaymath}
f(\alpha^{\scriptscriptstyle<-1>},t)-1=[f(\alpha,t)-1]^{\scriptscriptstyle<-1>}.
\end{displaymath}
The umbra $\alpha^{\scriptscriptstyle<-1>}$ is named the
\textit{compositional inverse} of $\alpha$ (see \cite{DiNaSen1}).
The \textit{$\alpha$-derivative} umbra, deeply studied in
{\cite{DiNaSen3}}, is an auxiliary umbra
$\alpha_{\scriptscriptstyle D}$ whose moments satisfies the
identity
\begin{displaymath}
(\alpha_{\scriptscriptstyle D})^i\simeq
\partial_{\alpha}\alpha^i\simeq i \alpha^{i-1},\quad i=1,2,\ldots.
\end{displaymath}
We obtain
\begin{displaymath}
f(\alpha_{\scriptscriptstyle D},t)\simeq\sum_{i\geq
0}(\alpha_{\scriptscriptstyle D})^i\frac{t^i}{i!}\simeq\sum_{i\geq
0}\alpha^{i-1}\frac{t^i}{(i-1)!}\simeq 1+t\,f(\alpha,t).
\end{displaymath}
Let $\bar{\rho}$ be a polynomial umbra with moments $\bar{\rho}^i\simeq
i!\mathcal{PF}_i$, then
\begin{displaymath}
f(\bar{\rho},t)=\mathcal{PF}(t).
\end{displaymath}
Identity (\ref{St1}) provides
\begin{equation}\label{id:PF vs H 3 umbral}
\bar{\rho}_{\scriptscriptstyle
D}\equiv{(-1\punt\bar{\vartheta})_{\scriptscriptstyle
D}}^{\scriptscriptstyle<-1>}.
\end{equation}
By means of (\ref{id:park-pol vs PFn}) and (\ref{id:PF vs H 3
umbral}) we have
$$n!V_n(\bar{\vartheta}_1,\ldots,\bar{\vartheta}_n)\simeq\frac{1}{n+1}[{(-1\punt\bar{\vartheta})_{\scriptscriptstyle
D}}^{\scriptscriptstyle<-1>}]^{n+1}.$$
It is not too difficult to show that such an equivalence will be true even if
we replace the umbra $\bar{\vartheta}$ with another umbra
$\alpha$. That is
\begin{equation}\label{id:1}
n!V_n(\alpha_1,\ldots,\alpha_n)\simeq\frac{1}{n+1}[{(-1\punt\alpha)_{\scriptscriptstyle
D}}^{\scriptscriptstyle<-1>}]^{n+1},
\end{equation}
for all $\alpha\in A$.
\begin{thm}
If $\bar{\vartheta}_1,\ldots,\bar{\vartheta}_n$ are $n$
uncorrelated umbrae similar to $\bar{\vartheta}$, $k$ is a
positive integer and $V_n(\mathbf{x})$ is the $n$-volume
polynomial $(\ref{def:park-polin})$, then
$$V_n(k\punt\bar{\vartheta}_1,\ldots,k\punt\bar{\vartheta}_n)\simeq \mathcal{PF}_n^{(k)}.$$
\end{thm}
\begin{proof}
Let $\bar{\rho}^{(k)}$ denote an umbra such that
$(\bar{\rho}^{(k)})^n\simeq n!\mathcal{PF}_n^{(k)}$. From (\ref{St2}) we have
$$n!\mathcal{PF}_n^{(k)}\simeq(\bar{\rho}^{(k)})^n\simeq\frac{1}{n+1}\{[(-k\punt\bar{\vartheta})_{\scriptscriptstyle
D}]^{\scriptscriptstyle<-1>}\}^{n+1}.$$
Finally, since $-1\punt
k\punt\bar{\vartheta}\equiv-k\punt\bar{\vartheta}$, from
(\ref{id:1}) we gain

$$n!V_n(k\punt\bar{\vartheta}_1,\ldots,k\punt\bar{\vartheta}_n)\simeq\frac{1}{n+1}[{(-k\punt\bar{\vartheta})_{\scriptscriptstyle
D}}^{\scriptscriptstyle<-1>}]^{n+1}\simeq n!\mathcal{PF}_n^{(k)},$$
from which the theorem is proved.
\end{proof}
\section{Parking functions of type B}
As shown by Biane \cite{Bia1} the lattice of noncrossing
partitions can be embedded into the Cayley graph of the symmetric
group. Reiner
\cite{Rei} has introduced a class of noncrossing partitions for
all classical reflection groups, that for type A reflection groups (i.e. symmetric groups) corresponds to $\mathcal{NC}_n$. In the case
of hyperoctahedral groups, that is type B reflection groups, the \textit{noncrossing partitions of type B}
are defined to be the noncrossing partitions of the set $[\pm
n]=\{-n,\ldots,-1,1,\ldots,n\}$ which are invariant under sign
change. Let denote by $\mathcal{NC}_n^{\scriptscriptstyle B}$ the
set of such partitions. This methods have been generalized by
Brady \cite{Br}, Brady and Watt \cite{BW} and Bessis \cite{Bes}. A
poset $\mathcal{NC}(W)$ of noncrossing partitions can be defined
for finite Coxeter systems $(W,S)$ of each type. In particular,
these posets have a nice description in terms of the length
function $\ell_{\scriptscriptstyle T}$ defined with respect to their
respective sets $T$ of reflections. The reader may refer \cite{Arm} for an overview on the combinatorial
aspect of this beautiful subject. Stanley \cite{Stan1} has
noticed that maximal chains of $\mathcal{NC}_n^{\scriptscriptstyle
B}$ are labeled by all sequences in $[n]^n$ each occurring once. From this analogy with parking functions,
he has named them \textit{parking functions of type B}. Biane
\cite{Bia} has completed the picture by showing that parking
functions of length $n$  of type A (i.e. classical ones) and B correspond to
factorizations of the cycles $(1\ldots n+1)$ and
$(-n\ldots-1\,1\ldots n)$ respectively into products of
reflections. The parking function module of type $B$ can be defined by considering the standard action of $\frak{S}_n$ on $[n]^n$. Following Stanley, we denote by
$\mathcal{PF}_n^{\scriptscriptstyle B}$ its Frobenius
characteristic.
\begin{thm}[\cite{Stan1}]
Let $\mathcal{PF}^{\scriptscriptstyle B}(t)=1+\sum_{n\geq 1}\mathcal{PF}_n^{\scriptscriptstyle B}t^n$, then we have
\begin{equation}\label{id:PFB}
\mathcal{PF}_n^{\scriptscriptstyle B}=[t^n]H(t)^n,
\end{equation}
where $[t^n]$ means taking the coefficient of $t^n$ in the power series.
\end{thm}
We will define a type B Abel polynomial, denoted by
$B_{n}(x,\alpha)$, which plays a role analogous to
$A_{n}(x,\alpha)$ for the parking function module of type B. It is
simply obtained by dividing $A_{n+1}(x,\alpha)$ by $x$, that is
$$B_{n-1}(x,\alpha)=(x-n\punt\alpha)^{n-1}.$$
\begin{thm}
Let $\mathcal{PF}^{\scriptscriptstyle B}$ be the Frobenius
characteristic of the parking function module of type B and
$\bar{\vartheta}$ be the polynomial umbra defined in (\ref{def:
vartheta}). Then we have
$$n!\mathcal{PF}_n^{\scriptscriptstyle B}\simeq B_n(-1\punt\bar{\vartheta},-1\punt\bar{\vartheta})\simeq(n\punt\bar{\vartheta})^n.$$
\end{thm}
\begin{proof}
If $\bar{\rho}^{\scriptscriptstyle B}$ is a polynomial umbra such that $(\bar{\rho}^{\scriptscriptstyle B})^i\simeq i!\mathcal{PF}^{\scriptscriptstyle B}$, then from (\ref{id:PFB}) we have
$$(\bar{\rho}^{\scriptscriptstyle B})^n\simeq(n\punt\bar{\vartheta})^n.$$
Finally, since
$-1\punt\bar{\vartheta}+(n+1)\punt\bar{\vartheta}\equiv
n\punt\bar{\vartheta}$, then we have
$B_n(-1\punt\bar{\vartheta},-1\punt\bar{\vartheta})\simeq(n\punt\bar{\vartheta})^n$
and the proof is completed.
\end{proof}
Of course, if we introduce the polynomial $V_n^{\scriptscriptstyle B}(\mathbf{x})$ such that
$$V_n^{\scriptscriptstyle B}(\mathbf{x})=\frac{1}{n!}\sum_{\mathbf{p}\in[n]^n}\mathbf{x}_{\mathbf{p}}=\sum_{\mu\vdash n}\frac{(n)_{\ell(\mu)}}{m(\mu)!\mu!}\mathbf{x}^{\mu},$$
then we can complete the analogy with the results in the previous section.
\begin{thm}
The following umbral equivalence holds:
$$n!V_n^{\scriptscriptstyle B}(\bar{\vartheta}_1,\ldots,\bar{\vartheta}_n)\simeq B_{n}(-1\punt\bar{\vartheta},-1\punt\bar{\vartheta}).$$
\end{thm}
\begin{proof}
It follows by simple computations.
\end{proof}
\end{document}